\documentclass[a4paper,11pt,final,oneside]{amsart}			
\usepackage[utf8]{inputenc}						
\usepackage[T1]{fontenc}		  			    
\usepackage[danish,english]{babel}			
\usepackage{amsmath,amsfonts,amssymb,amsthm}	
\usepackage{enumerate}				
\usepackage{graphicx}			
\usepackage{tikz}
	\usetikzlibrary{matrix,arrows}
\usepackage{comment}
\usepackage{mathrsfs}
\usepackage{fixme}
\usepackage{mathtools}
\usepackage{amsthm}
\usepackage[final,colorlinks=true,linkcolor=black,citecolor=black]{hyperref}

\usepackage[%
	hmargin={34mm,34mm},%
	vmargin={41.6mm,41.6mm}%
]{geometry}

\newcommand{\interior}[1]{%
  {\kern0pt#1}^{\mathrm{o}}%
}

\newcommand{\C}{\mathbb{C}}

\newcommand{\R}{\mathbb{R}}
\DeclareMathOperator{\THH}{THH}

\DeclareMathOperator{\THR}{THR}

\DeclareMathOperator{\Gal}{Gal}

\newcommand{\id}{\mathrm{id}}
\DeclareMathOperator*{\hocolim}{hocolim}

\DeclareMathOperator*{\colim}{colim}
\DeclareMathOperator{\Map}{Map}

\DeclareMathOperator{\Ob}{Ob}
\newcommand{\ob}{\Ob}

\DeclareMathOperator{\conn}{conn}
\newcommand{\Top}{\mathsf{Top}}

\newcommand{\incl}{\mathrm{incl}}

\newcommand{\op}{\mathrm{op}}
\DeclareMathOperator{\ind}{ind}

\newcommand{\sd}{\mathrm{sd}}
\newcommand{\Z}{\mathbb{Z}}

\theoremstyle{plain}
\newtheorem{theorem}{Theorem}[section]
\newtheorem{lemma}[theorem]{Lemma}

\newtheorem{prop}[theorem]{Proposition}

\theoremstyle{definition}
\newtheorem{definition}[theorem]{Definition}
\newtheorem{example}[theorem]{Example}
\newtheorem{remark}[theorem]{Remark}
\newtheorem{ass}[theorem]{Assumptions}

\theoremstyle{plain}   
\newtheorem{bigthm}{Theorem}   

\author{Amalie Høgenhaven}
\date{\today}
\thanks{Assistance from DNRF Niels Bohr Professorship of Lars Hesselholt is
gratefully acknowledged}

\newlength{\TMP}

\begin{document}

\title[On the geometric fixed points]{On the geometric fixed points of real topological Hochschild homology}

\maketitle

\begin{abstract}
We compute the component group of the derived $G$-geometric fixed points of the real topological Hochschild homology of a ring with anti-involution, where $G$ denotes the group $\Gal(\mathbb{C}/\mathbb{R})$ of order 2. 

\end{abstract}

\section*{Introduction}
Recently, Hesselholt and Madsen defined real topological Hochschild homology in~\cite{HM15} using a dihedral variant of Bökstedt's model in \cite{Bo}. The real topological Hochschild homology functor takes a ring $R$ with an anti-involution $\alpha$, that is a ring isomorphism $\alpha: R^{\op} \to R$ such that $\alpha^2=\id$, and associates an $O(2)$-equivariant orthogonal spectrum $\THR(R, \alpha)$. 

Real topological Hochschild homology was introduced to fit in the framework of real algebraic $K$-theory, which was also defined by Hesselholt and Madsen in \cite{HM15}. Real algebraic $K$-theory associates a $G$-spectrum $KR(R,\alpha)$ to a ring $R$ with an anti-involution $\alpha$, where $G$ is the Galois group $\Gal(\mathbb{C} / \mathbb{R})$. The underlying non-equivariant spectrum is weakly equivalent to the usual $K$-theory spectrum of $R$, and the $G$-fixed point spectrum is weakly equivalent to the Hermitian $K$-theory spectrum of $(R,\alpha)$, as defined by Karoubi in \cite{Kar}, when $2$ is invertible in the ring. Furthermore, there is a $G$-equivariant trace map
\[
\operatorname{tr}: KR(R,\alpha) \to \THR(R,\alpha).
\]

The classical trace maps are often highly non-trivial and several calculations in algebraic $K$-theory have been carried out using the trace, or more precisely the refinement of the trace to topological cyclic homology, see \cite{BHM}. Classical calculations using trace methods often rely on a good understanding of $\pi_*\THH(R)^{C_{r}}$. In order to make the equivariant trace an efficient computational tool, we must understand the dihedral fixed points $\pi_*\THR(R, \alpha)^{D_{r}}$ and, in particular, the components $ \pi_0\THR(R, \alpha)^{D_{r}}$. As a first step in this direction, we calculate the group of components of the derived $G$-geometric fixed points. 

The orthogonal spectrum $\THR(R,\alpha)$ is cyclotomic, which means that its derived $C_r$-geometric fixed points mimic the behavior of the $C_r$-fixed points of a free loop space $\mathcal{L} X$ of a $G$-space $X$, see \cite[Prop. 1.5]{HM97} and \cite[Sect. 3.3]{Hoe} for details. In particular, it implies that the spectrum of $C_r$-geometric fixed points of $\THR(R, \alpha)$ resembles $\THR(R, \alpha)$ itself. The derived $G$-geometric fixed point, however, behave differently. In the analogy with the free loop space, the derived $G$-geometric fixed points of $\THR(R,\alpha)$ corresponds to the $G$-fixed points of $\mathcal{L} X$, and we briefly investigate how the latter behaves. 

We let $\mathcal{L} X=\Map(\mathbb{T}, X)$ be the free loop space of a $G$-space $X$. The group $O(2)$ acts on $\mathbb{T}$ by multiplication and complex conjugation and we view $X$ as an $O(2)$-space with trivial $\mathbb{T}$-action. The free loop space becomes an $O(2)$-space by the conjugation action. Let $\mathbb{T} \subset \mathbb{C}$ denote the circle group. The group $G=\Gal(\mathbb{C}/\mathbb{R})$ acts on $\mathbb{T}$ and we let $O(2)$ denote the semi-direct product $O(2)= \mathbb{T} \rtimes G$. If $r$ is a natural number, then we let 
\[\rho_r: O(2) \to O(2)/C_r
\] 
denote the root isomorphism given by $\rho_r(z)=z^{\frac{1}{r}}C_r$ if $z\in \mathbb{T}$ and $\rho_r(x)=x$ if $x\in G$. The map that takes a loop to the $r$-fold concatenation with itself,
\begin{align*}
p_r :  \mathcal{L} X \to \rho_r^*(\mathcal{L} X)^{C_r}, \quad p_r( \gamma ) = \gamma \star \cdots \star \gamma, 
\end{align*}
is an $O(2)$-equivariant homeomorphism, but the $G$-fixed space of the free loop space looks very different from the loop space itself.  Indeed, if $\omega \in G$ is complex conjugation then we have a homeomorphism
\[
\Map((I, \partial I), (X, X^G)) \to (\mathcal{L} X)^G,  \quad \gamma \mapsto  (\omega \cdot \gamma) \star \gamma. 
\]

The content of this paper is organized as follows. In Section 1 and 2 we review the construction of the orthogonal $O(2)$-spectrum $\THR(R,\alpha)$, and observe that, if $R$ is a commutative ring, then $\THR(R,\alpha)$ has the homotopy type of a commutative $O(2)$-ring spectrum. 

In section 3 we prove the main theorem of this paper. In order to state the theorem, we let $R$ be a ring with an anti-involution $\alpha$, which is a ring isomorphism $\alpha: R^{\op} \to R$ such that $\alpha^2=\id$, and we let $N: R \to R^{\alpha}$ denote the norm map 
\[N(r)=r+\alpha(r).\] 

\begin{bigthm}\label{A}
Let $R$ be a ring with an anti-involution $\alpha$. There is an isomorphism of abelian groups
\[
\pi_0 \big( (\THR(R, \alpha)^c)^{gG} \big) \cong (R^{\alpha}/N(R)\otimes_{\Z} R^{\alpha}/N(R))/I,
\]
where $I$ denotes the subgroup generated by the elements $\alpha(s)rs \otimes t-r \otimes st\alpha(s)$ for all $s \in R$ and $r,t \in R^{\alpha}$. 
\end{bigthm}
The identification of the component group can be rewritten as
\[
\pi_0 \big( ( \THR(R, \alpha)^c)^{gG} \big) \cong R^{\alpha}/N(R)\otimes_R R^{\alpha}/N(R).
\]
where we view the group $R^{\alpha}/N(R)$ as a right resp. left $R$-module via the actions 
\[
 x \cdot r= \alpha(r)xr \quad \text{and} \quad r\cdot x= rx \alpha(r).
\] 
We note that the $G$-geometric fixed points vanishes if $2$ is invertible in $R$, since the norm map surjects onto the fixed points of the anti-involution in this case: If $x\in R^{\alpha}$, then $N(\frac{1}{2}x)=x$.

We end this introduction by stating some immediate consequences of \mbox{Theorem \ref{A}}. 

If $R$ is a commutative ring, then the components of the $C_{r}$-fixed points and the components of the $D_{r}$-fixed points have ring structures. The component ring of the $C_{p^n}$-fixed points, $\pi_0\THH(R)^{C_{p^n}}$, is completely understood when $p$ is a prime. Hesselholt and Madsen prove in \cite{HM97} that there is a canonical ring isomorphism identifying $ \pi_0\THH(R)^{C_{p^n}}$ with the $p$-typical Witt vectors of length $n+1$. The classical construction of the Witt vectors can be understood as a special case of a construction which can be defined relative to any given profinite group, as done by Dress and Siebeneicher in \cite{DS}. Furthermore, the $p$-typical Witt vectors of length $n+1$ are exactly the Witt vectors constructed relative to the group $C_{p^n}$. In other words, 
\[  \pi_0\THH(R)^{C_{p^n}} \cong W_{C_{p^n}}(R).
\] 
If $R$ is commutative, then the identity defines an anti-involution on $R$ and it is tempting to guess that the ring $\pi_0\THR(R, \id)^{D_{p^n}} $ can be identified with the Witt vectors $W_{D_{p^n}}(R)$. However, Theorem \ref{A} tells us that this is not the case; see Remark \ref{witt}.

When $R$ is a commutative ring, the components of the derived $G$-geometric fixed points of $\THR(R,\alpha)$ have a ring structure. If $R$ is a commutative ring with the identity serving as anti-involution, then Theorem \ref{A} implies the functor 
\[
R \mapsto  \pi_0 \big(( \THR(R, \id)^c)^{gG} \big).
\]
considered as a functor from the category of commutative rings to the category of sets, is not representable, since the functor does not preserve finite products, see Remark \ref{limit}. This rules out the possibility that $\big(( \THR(R, \id)^c)^{gG} \big)$ is a ring of Witt vectors as defined by Borger in \cite{Borger}.

\begin{remark}
Throughout this paper, $\mathbb{T} \subset \mathbb{C}$ denotes the circle group and $G$ is the group $\Gal(\C/\R)=\{1,\omega\}$. The group $G$ acts on $\mathbb{T} \subset \mathbb{C}$ and $O(2)$ is the semi-direct product $O(2)= \mathbb{T} \rtimes G$. We let $C_r$ denote the cyclic subgroup of order $r$ and let $D_{r}$ denote the dihedral subgroup $C_r \rtimes G$ of order $2r$. 

By a space we will always mean a compactly generated weak Hausdorff space and all constructions are always carried out in this category. 
\end{remark}

\medskip
\noindent\textbf{Acknowledgments.} The author wishes to thank Lars Hesselholt for his guidance and for many valuable discussions. The author would also like to thank Martin Speirs, Dustin Clausen, Irakli Patchkoria and Kristian Moi for many useful conversations concerning the content of this paper.

\section{Real topological Hochschild homology}
A symmetric ring spectrum $X$ is a sequence of based spaces $X_0, X_1, \dots$ with a left based action of the symmetric group $\Sigma_n$ on $X_n$ and $\Sigma_n \times \Sigma_m$-equivariant maps $\lambda_{n,m} : X_n \wedge S^m  \to X_{n+m}$. Let $A$ be a symmetric ring spectrum with multiplication maps $\mu_{n,m}: A_n \wedge A_m \to A_{n+m}$ and unit maps $1_n: S^n \to A_n$. An anti-involution on $A$ is a self-map of the underlying symmetric spectrum $D: A \to A$, such that 
\[D^2=\id, \quad D_n \circ 1_n=1_n,\] 
and the following diagram commutes:
\begin{center}
\begin{tikzpicture}
  \matrix (m) [matrix of math nodes,row sep=2em,column sep=4em,minimum width=2em] {
 A_m \wedge A_n & A_{m} \wedge A_{n} \\
  & A_{n} \wedge A_m\\
   & A_{n+m} \\
    A_{m+n}    & A_{m+n}.\\};
  \path[-stealth]
      (m-1-1) edge node [above] {$D_m \wedge D_n$} (m-1-2)          
    (m-1-2) edge node [right] {$\gamma$} (m-2-2)
    (m-2-2) edge node [right] {$\mu_{n,m} $} (m-3-2)
    (m-3-2) edge node [right] {$\chi_{n,m}$} (m-4-2)
        (m-1-1) edge node [right] {$\mu_{m,n}$} (m-4-1)
    (m-4-1) edge node [above] {$D_{m+n}$} (m-4-2)
  
    ;
\end{tikzpicture}
\end{center}
Here $\gamma$ is the twist map and $\chi_{n,m}\in \Sigma_{n+m}$ is the shuffle permutation
\[\chi_{n,m}(i)= \left\{ \begin{array}{rl}
 i+m &\mbox{ if $1 \leq i \leq n$} \\
  i-n &\mbox{ if $n+1 \leq i \leq n+m.$}
       \end{array} \right. \]

Let $R$ be a unital, associative ring. Then $R$ determines a symmetric ring spectrum $HR$, called the Eilenberg MacLane spectrum of $R$, which can be constructed as follows. Let $S^1[-]:=\Delta^1[-]/\partial \Delta^1[-]$ denote the pointed simplicial circle and let $S^n[-]$ denote the pointed simplicial $n$-sphere defined as the $n$-fold smash product $S^1[-] \wedge \cdots \wedge S^1[-]$. The $n$th space of the spectrum $HR$ is the realization of the reduced $R$-linearization of the simplicial $n$-sphere:
\[
HR_n:= R(S^n)=\vert [k] \mapsto R[S^n[k]]/ R[*]\vert.
\]
Here $R[S^n[k]]$ is the free $R$-module generated by the $k$-simplices $S^n[k]$ and $R[*]$ is the sub-$R$-module generated by the basepoint $* \in S^n[k]$. The symmetric group $\Sigma_n$ acts by permutation of the smash factors of $S^n[-]$ and there are natural multiplication and unit maps
\[
\mu_{m,n}: HR_m \wedge HR_n \to HR_{m+n}, \quad 1_n: S^n \to HR_n,
\]
which are $\Sigma_m \times \Sigma_n$-equivariant and $\Sigma_n$-equivariant.

An anti-involution $\alpha$ on $R$ is a ring isomorphism $\alpha: R^{\text{op}} \to R$ such that $\alpha^2=\id$. If $\alpha$ is an anti-involution on $R$, then we also let $\alpha$ denote the induced anti-involution on the symmetric ring spectrum $HR$, which in spectral level $n$ is the geometric realization of the map of simplicial $R$-modules given by $r\cdot x \mapsto \alpha(r) \cdot x$ for $r\in R$ and $x\in S^n[k]$.

Given a symmetric ring spectrum with anti-involution $(A,D)$, the real topological Hochschild homology space $\THR (A,D)$ was defined in \cite{HM15} as the geometric realization of a dihedral space, and we briefly recall the notion of a dihedral object. The real topological Hochschild homology of a ring with anti-involution $(R,\alpha)$ is the real topological Hochschild homology of $(HR,\alpha)$, which we simply denote $\THR(R,\alpha)$. 

\begin{definition}\label{didef}
A dihedral object in a category $\mathcal{C}$ is a simplicial object 
\[X[-]: \triangle^{\op} \to \mathcal{C}\] 
together with dihedral structure maps $t_k, w_k: X[k] \to X[k]$ such that $t_k^{k+1}=\id$, $w_k^2=\id$, and $t_k w_k=t_k^{-1}\omega_k$. The dihedral structure maps are required to satisfy the following relations involving the simplicial structure maps:
\begin{align*}
d_l w_k=w_{k-1}d_{k-l}, &\quad s_l w_k=w_{k+1}s_{k-l} \ \text{ if }0\leq l\leq k, \\ 
d_l t_k=t_{k-1}d_{l-1}, &\quad s_l t_k=t_{k+1}s_{l-1} \ \ \ \text{ if } 0<l\leq k,  \\
 d_0t_k=d_k, &\quad s_0t_k=t_{k+1}^2s_k.
\end{align*}
\end{definition}
A simplicial object together with structure maps $t_k: X[k] \to X[k]$ satisfying the above relations is called a cyclic object and a simplicial object together with structure maps $w_k: X[k] \to X[k]$ satisfying the above relations is called a real object. The geometric realization of the simplicial space underlying a dihedral (resp. cyclic, resp. real) space carries an action by $O(2)$ (resp. $\mathbb{T}$, resp. $G$): See \cite{FL} for more details.

Let $I$ be the category with objects all non-negative integers. The morphisms from $i$ to $j$ are all injective set maps
\[\{1,\dots,i\}  \to \{1,\dots, j\}. \] 
The category $I$ has a strict monoidal product $+: I\times I \to I$ given on objects by addition and on morphisms by concatenation. We note that the initial object $0\in \Ob(I)$ serves as the identity for the monoidal product. For $i \in \Ob(I)$ we let $\omega_i: i \to i$ denote the morphism that reverses the order of the elements:
\[\omega_i(s)=i-s+1. \]
Given a morphism $\theta: i \to j$ we define the conjugate morphism $\theta^\omega$ by $\theta^{\omega}:=\omega_j \circ \theta\circ \omega_i^{-1}$.
We define a dihedral category $I[-]$ by letting $I[k]=I^{k+1}$ and defining cyclic structure maps $d_i: I[k] \to I[k-1]$, $s_i: I[k] \to I[k+1]$, and $t_k: I[k] \to I[k]$ on objects by
\begin{align*}
d_i(i_0,\cdots, i_k)&=(i_0,\cdots, i_i+i_{i+1},\cdots,i_k), \quad 0\leq i <k,  \\
d_i(i_0,\cdots, i_k)&=(i_k + i_0, \cdots,i_{k-1}), \quad \quad \quad \quad  i=k,  \\
s_i(i_0,\cdots, i_k)&=(i_0,\cdots, i_i, 0, i_{i+1},\cdots,i_k),\quad 0\leq i \leq k,  \\
t_k(i_0,\cdots, i_k)&=(i_k, i_0,\cdots, i_{k-1}), 
\end{align*}
and similarly on morphisms. The structure maps $w_k: I[k] \to I[k]$ is defined on a tuple of objects by
\begin{align*}
w_k(i_0,\dots,i_k)&=(i_0,i_k,i_{k-1},\dots,i_1)
\end{align*}
and on a tuple of morphisms by
\begin{align*}
w_k(\theta_0,\dots,\theta_k)&= (\theta_0^{\omega}, \theta_k^{\omega}\dots,\theta_1^{\omega}).
\end{align*}

Let $X$ be a pointed space with a pointed left $O(2)$-action and let $(A,D)$ be a symmetric ring spectrum with an anti-involution. Let $G(A)_X^k: I^{k+1} \to \Top_*$ denote the functor given on objects by
\[G(A)_X^k(i_0,\dots, i_k)= \Map \Big( S^{i_0} \wedge \cdots \wedge S^{i_k} ,  A_{i_0}\wedge \cdots \wedge A_{i_k} \wedge  X  \Big). \]
We will almost always omit the $A$ and simply write $G_X^{k+1}$ if there is no confusion about which spectrum $A$ is used in the construction of the functor. The functor is defined on morphisms using the structure maps of the spectrum; see \cite{HM97} or \cite[Sect. 4.2.2]{DGM}. We define a dihedral space by setting
\[
\THR(A,D;X)[k]:= \hocolim_{I^{k+1}} G_X^k
\]
with simplicial structure maps as described in \cite{HM97} or \cite[Sect. 4.2.2]{DGM}. Let $\omega_i \in \Sigma_i$ be the permutation given by $\omega_i(s)=i-s+1$. We let 
\[ t_k':G_X^k \Rightarrow  G_X^k \circ t_k, \quad w_k':G_X^k \Rightarrow  G_X^k \circ w_k \]
be the natural transformations which at $(i_0, \dots, i_k) \in \Ob(I^k)$ are defined by the following commutative diagrams
 \[
 \begin{tikzpicture}
\node(a){$   S^{i_0}\wedge\cdots\wedge S^{i_k}$}; 
\node(b)[right of = a,node distance = 6cm]{$A_{i_0}\wedge\cdots\wedge A_{i_k}\wedge X$};
\node(e)[below of = a,node distance = 1.5cm]{$  S^{i_k}\wedge S^{i_0} \wedge \cdots\wedge S^{i_{k-1}}$};
\node(f)[right of = e,node distance = 6cm]{$A_{i_k}\wedge A_{i_0} \wedge \cdots\wedge A_{i_{k-1}}\wedge X$};
\draw[->](a) to node [above]{$f$} (b);
\draw[->](e) to node [left]{$\tau^{-1}$} (a);
\draw[->](b) to node [right]{$\tau \wedge X$} (f);
\draw[->](e) to node [above]{$t_k'(f)$} (f);
\end{tikzpicture}
 \]
and 
 \[
 \begin{tikzpicture}
\node(a){$ S^{i_0}\wedge S^{i_1} \wedge \cdots\wedge S^{i_k}$}; 
\node(b)[right of = a,node distance = 6cm]{$A_{i_0}\wedge A_{i_1}\wedge \cdots\wedge A_{i_k}\wedge X$};
\node(c)[below of = a,node distance = 1.5cm]{$  S^{i_0}\wedge S^{i_k} \wedge \cdots\wedge S^{i_1}$};
\node(d)[right of = c,node distance = 6cm]{$A_{i_0} \wedge A_{i_1} \wedge\cdots\wedge A_{i_k}\wedge X$};
\node(e)[below of = c,node distance = 1.5cm]{$  S^{i_0}\wedge S^{i_k} \wedge \cdots\wedge S^{i_1}$};
\node(f)[right of = e,node distance = 6cm]{$A_{i_0}\wedge A_{i_1}\wedge\cdots\wedge A_{i_k}\wedge X$};
\node(g)[below of = e,node distance = 1.5cm]{$  S^{i_0}\wedge S^{i_k} \wedge \cdots\wedge S^{i_1}$};
\node(h)[right of = g,node distance = 6cm]{$A_{i_0} \wedge A_{i_k} \wedge\cdots\wedge A_{i_1}\wedge X$};

\draw[->](a) to node [above]{$f$} (b);
\draw[->](g) to node [above]{$w_k'(f)$} (h);

\draw[->](g) to node [left]{$\omega_0 \wedge \cdots \wedge \omega_k $} (e);
\draw[->](e) to node [left]{$\id$} (c);
\draw[->](c) to node [left]{$\upsilon$} (a);

\draw[->](b) to node [right]{$\omega_{i_0} \wedge \cdots \wedge \omega_{i_k} \wedge \id $} (d);
\draw[->](d) to node [right]{$D_{i_0}  \wedge \cdots \wedge D_{i_k} \wedge \id $} (f);
\draw[->](f) to node [right]{$\upsilon$} (h);

\end{tikzpicture}
 \]
where $\tau$ cyclically permutes the smash factors to the right, and $\upsilon$ fixes the first smash factor and reverses the order of the rest. The structure maps are given as the compositions of the maps induced by the natural transformations and the canonical maps:
\begin{align*}
t_k :\hocolim_{I^{k+1}} \ G_X^k &\xrightarrow{t_k'} \hocolim_{I^{k+1}}  \ G_X^{k}  \circ t_k \xrightarrow{\ind_{t_k}}  \hocolim_{I^{k+1}}  \ G_X^{k} , \\
w_k :\hocolim_{I^{k+1}}  \ G_X^k &\xrightarrow{w_k'} \hocolim_{I^{k+1}} \ G_X^{k}  \circ w_k \xrightarrow{\ind_{w_k}}  \hocolim_{I^{k+1}} \ G_X^{k}. 
\end{align*}
We have defined a dihedral space and we let $ \THR(A, D;X)$ denote the realization
 \[ 
 \THR(A, D;X):= \Big \vert [k] \mapsto \THR(A, D;X)[k]  \Big \vert.
 \]
The space $\THR(A, D;X)$ is in fact an $O(2) \times O(2)$-space, where the action by the first factor comes from the dihedral structure and the action by the second factor comes from the $O(2)$-action on $X$. We are interested in the space $\THR(A, D;X)$ with the diagonal $O(2)$-action.

\begin{remark}\label{sub}
Let $\triangle: O(2) \to O(2) \times O(2)$ denote the diagonal map. The tool available for investigating the fixed point space $(\triangle^*\THR(A, D; X))^{G}$ is Segal's real subdivision constructed in \cite[Appendix A1]{Seg}. We briefly recall the real subdivision functor $\sd^e$ and refer to Segal's paper for details. 

Let $X[-]$ be a dihedral space. There is a (non-simplicial) homeomorphism
\[D^e: \vert  \sd^eX[-]\vert \to \vert X[-] \vert,\] 
where $\sd^e X[-]$ is the simplicial space with $k$-simplices $\sd^e X[k]=X[2k+1]$ and simplicial structure maps, for $0\leq i \leq k$, given by
\begin{align*}
(d_i)^e &: \sd^e X[k] \to \sd^e X[k-1], \quad (d_i)^e =d_i \circ d_{2k+1-i,} \\
(s_i)^e &: \sd^e X[k] \to \sd^e X[k+1], \  \quad (s_i)^e = s_i \circ s_{2k+1-i}.
\end{align*}
The simplicial set $\sd^e X[-]$ has a simplicial $G$-action which in simplicial level $k$ is generated by $w_{2k+1}$.  Thus the realization inherits a $G$-action. The advantage of real subdivision is that the homeomorphism $D^e$ is $G$-equivariant. In particular, it induces a homeomorphism
\[\vert  \sd^eX[-]^G\vert \to \vert X[-] \vert^G.\] 
\end{remark}
 
Note that
\[
\sd^e \THR(A, D ;X)[k]=\hocolim_{I^{2k+2}} \ G_{X}^{2k+1}=\big \vert [n] \mapsto \bigvee_{\underline{i}_0 \to \cdots \to \underline{i}_n} G_{X}^{2k+1} (\underline{i}_0)\big \vert,
\]
where $\underline{i} \in\Ob(I^{2k+1})$. The $G$-action on $X$ gives rise to a natural transformation 
\[
X_{\omega}: G_{X}^{2k+1} \circ w_{2k+1}  \Rightarrow G_{X}^{2k+1} \circ w_{2k+1}.
\] 
The diagonal $G$-action is generated by the simplicial operator which takes the summand indexed by $\underline{i}_0 \to \cdots \to \underline{i}_n$ to the one indexed by $w_{2k+1}(\underline{i}_0) \to \cdots \to w_{2k+1}(\underline{i}_n)$ via 
\begin{align} \label{t1}
G_{X}^{2k+1} (\underline{i}_0) \xrightarrow{w_{2k+1'}} G_{X}^{2k+1}\circ w_{2k+1}(\underline{i}_0) \xrightarrow{X_\omega} G_{X}^{2k+1}\circ  w_{2k+1} (\underline{i}_0).
\end{align}
In particular if a $k$-simplex is fixed, then it must belong to a summand whose index is fixed under the functor $w_{2k+1}$. Such an index consists of objects of the form
\begin{align}\label{t3}
(i_0,i_1,\dots,i_k,i_{k+1},i_k,\dots, i_1)
\end{align}
and morphisms of the form 
\begin{align} \label{t4}
(\theta_0,\theta_1,\dots,\theta_k,\theta_{k+1},\theta^{\omega}_k,\dots, \theta^{\omega}_1) 
\end{align}
where
$\theta_0= \theta_0^{\omega}$ and $\theta_{k+1}=\theta_{k+1}^{\omega}$. Let $I^{G}$ denote the subcategory of $I$ with the same objects and all morphisms $\theta$ which satisfies $\theta^\omega=\theta$. Let 
\[
\triangle^e : I^{G} \times I^{k} \times  I^{G}  \to I^{2k+2}
\]
denote the ``diagonal'' functor which maps a tuple $(i_0, i_1,\dots, i_k,i_{k+1})$ to the tuple \eqref{t3} and a tuple of morphisms $(\theta_0, \theta_1,\dots, \theta_k,\theta_{k+1})$ to the tuple \eqref{t4}. The natural transformation \eqref{t1} restricts to a natural transformation from $G^{2k+1}_{X}\circ \triangle^e$ to itself, hence $G$ acts on $G^{2k+1}_{X}\circ \triangle^e$ through natural transformations. Since geometric realization commutes with taking fixed points of the finite group $G$ by \cite[Cor. 11.6]{May72}, we obtain the following lemma.

\begin{lemma}\label{fixed points}
The canonical map induces a homeomorphism
\begin{align*}
\hocolim_{I^{G} \times I^k \times I^{G}}  (G^{2k+1}_{X}\circ \triangle^e)^{G} \xrightarrow{\cong} \Big(\hocolim_{I^{2k+2}}  G^{2k+1}_{X} \Big)^{G}.
\end{align*}
\end{lemma}

The $G$-action on $G_{X}^{2k+1} \circ \triangle^e$ at $(i,n_1, \dots, n_k, j )\in \ob\big( I^{G} \times I^k \times I^G\big)$ can be described as follows. The image of the functor is the mapping space:
\begin{align*}
G_{X}^{2k+1} \circ \triangle^e(i,n_1, \dots, n_k, j )= \Map\big(\overline{S}, \overline{A}\wedge X \big) ,
\end{align*}
where 
\begin{align*}
\overline{S} &= S^{i}\wedge S^{n_1} \wedge \cdots \wedge S^{n_k}  \wedge S^j \wedge S^{n_k} \wedge \cdots \wedge S^{n_1}, \\
\overline{A}&= A_{i}\wedge A_{n_1} \wedge \cdots \wedge A_{n_k}  \wedge A_j \wedge A_{n_k} \wedge \cdots \wedge A_{n_1} .
\end{align*}
The spaces above are $G$-spaces: The non-trivial element $\omega \in G$ fixes the first smash factor and reverses the order of the remaining factors, then acts by the permutation $w_{i} \in \Sigma_{i}$ factor-wise. On the space $\overline{A}$, $\omega$ further acts by the anti-involution $D_i$ factor-wise. The space $\overline{A} \wedge X$ is given the diagonal action and finally $G$ acts on the mapping space by the conjugation action.

We will need a $G$-equivariant version of Bökstedt's Approximation Lemma as proven by Dotto; see \cite[4.3.2]{Do}. We call a map of $G$-spaces $f: Z \to Y$ $n$-connected if $f^K: Z^K \to Y^K$ is $n$-connected for $K\in \{e,G\}$.

\begin{prop}[Equivariant Approximation Lemma] \label{EAL} 
 Let $(R,\alpha)$ be a ring with an anti-involution $\alpha$ and let $V$ be a finite dimensional real $G$-representation. Given $n\geq 0$, there exists $ N\geq 0$ such that the $G$-equivariant inclusion 
\[
G(R)^{2k+1}_{S^V} \circ \triangle^e (i) \hookrightarrow \hocolim_{I^{2k+2}}  G(R)^{2k+1}_{S^V} 
\]
is $n$-connected for all $i \in \ob( I^{G} \times I^k \times I^G) $ coordinate-wise bigger than $N$.
\end{prop}

\section{The real topological Hochschild homology spectrum}
The space $\THR(A,D;S^0)$ is the $0$th space of a fibrant orthogonal  $O(2)$-spectrum in the model structure based on the family of finite subgroups of $O(2)$, see Proposition \ref{fibrant}. Furthermore, if $A$ is a commutative symmetric ring spectrum, then $\THR(A,D)$ has the homotopy type of an $O(2)$-ring spectrum.  

In the classical setup one needs certain connectivity assumptions on the spectrum $A$ to ensure that $\THH(A)$ has the correct homotopy type. We likewise need some connectivity assumption on $(A,D)$. For an integer $n$ we let $\lceil \frac{n}{2} \rceil $ denote the ceiling of $\frac{n}{2}$. Throughout this section we make the following assumptions on $(A,D)$:

\begin{ass}\label{concond} 
Let $(A,D)$ be a symmetric ring spectrum with anti-involution. We assume that $A_n$ is $(n-1)$-connected  as a non-equivariant space and that $\left(A_n\right)^{D_n \circ \omega_n}$ is $\left(\lceil \frac{n}{2} \rceil -1 \right)$-connected. Furthermore we assume that there exists a constant $\epsilon \geq 0$ such that the structure map $\lambda_{n,m}: A_n \wedge S^m \to A_{n+m}$ is $(2n+m- \epsilon)$-connected as a map of non-equivariant spaces and such that the restriction of the structure map $\lambda_{n,m}: A_n^{D_n \circ \omega_n} \wedge \left(S^m\right)^{\omega_m} \to \left(A_{n+m}\right)^{D_{n+m}\circ (\omega_n \times \omega_m)}$ is $(n+\lceil \frac{m}{2} \rceil -\epsilon)$-connected.
\end{ass}

By an $O(2)$-representation, we will mean a finite dimensional real inner product space on which $O(2)$ acts by linear isometries. We fix a complete $O(2)$-universe $\mathcal{U}$ and work in the category of orthogonal $O(2)$-spectra indexed on $\mathcal{U}$ as defined in \mbox{\cite[Chapter II.4]{MM}.} Let $V \subset \mathcal{U}$ be a finite $O(2)$-representation. Let
\[\THR(A,D)(V)=\triangle^* \THR(A,D; S^V), \]
where $\triangle: O(2) \to O(2) \times O(2)$ is the diagonal map. The orthogonal group $O(V)$ acts on $\THR(A,D)(V)$ through the sphere $S^V$. It is straightforward to construct spectral structure maps
\[\sigma_{V,W}:\THR(A,D)(V)  \wedge S^W \to \THR(A,D)(V\oplus W),\]
see \cite{GH} or \cite{Hoe}. The family of $O(V) \rtimes O(2)$-spaces $\THR(A,D)(V)$ together with the maps $\sigma_{V,W}$ defines an orthogonal $O(2)$-spectrum indexed on $\mathcal{U}$, which is denoted $\THR(A,D)$. The following result is proven in \cite[Prop. 3.6]{Hoe}.

\begin{prop} \label{fibrant}
If $V$ and $W$ are finite $O(2)$-representations, then the adjoint of the structure map
\[
\widetilde{\sigma}_{V,W}: \THR(A,D)(V) \to \Map(S^{W} , \THR(A,D)(V\oplus W)) 
\]
induces a weak equivalence on $H$-fixed points for any finite subgroup $H \leq O(2)$. 
\end{prop}

When $A$ is a commutative symmetric ring spectrum, $\THH(A)$ is a $\mathbb{T}$-ring spectrum, though we must change foundations and work in the category of symmetric orthogonal $\mathbb{T}$-spectra to display this structure; see \cite{Ho01}. The multiplicative and unital structure maps as described in \cite[Appendix]{GH} are compatible with the added $G$-action. Thus when $A$ is commutative, $\THR(A,D)$ is a symmetric orthogonal $O(2)$-ring spectrum. We briefly recall the construction and refer to \cite[Appendix]{GH} for details. Let $(n)$ denote the finite ordered set $\{1,\dots,n\}$ and let $I^{(n)}$ denote the product category. There is a functor
\[
\sqcup_n :I^{(n)} \to I
\]
given by addition of objects and concatenation of morphisms according to the order of $(n)$. Let $G_X^{k,(n)}$ denote the composite $G_X^k \circ (\sqcup_n )^{k+1}$. There is a dihedral space $\THR^{(n)}(A,D;X)[-]$ with $k$-simplices the homotopy colimit
\[
\THR^{(n)}(A,D;X)[k]:= \hocolim_{\big(I^{(n)}\big)^{k+1}} G_X^{k,(n)},
\]
and with cyclic structure maps constructed as for $\THR(A,D;X)[-]$ with minor adjustments.
We define 
\[w_k^{(n)}:\big( I^{(n)}\big)^{k+1} \to \big(I^{(n)}\big)^{k+1}\] 
on objects by
\[
w_k^{(n)}\big( (i_{01}, \dots , i_{0n}), \dots , (i_{k1}, \dots , i_{kn}) \big)=\big( (i_{01}, \dots , i_{0n}),(i_{k1}, \dots , i_{kn}), \dots , (i_{11}, \dots , i_{1n})  \big)
\]  
and on morphisms by
\[
w_k^{(n)}\big( (\alpha_{01}, \dots , \alpha_{0n}), \dots , (\alpha_{k1}, \dots , \alpha_{kn}) \big)= \big( (\alpha^{\omega}_{01}, \dots , \alpha^{\omega}_{0n}), \dots , (\alpha^{\omega}_{11}, \dots , \alpha^{\omega}_{1n})  \big)
\]  
Furthermore we define the natural transformation 
\[
\big(w_k^{(n)}\big)': G_X^{k,(n)} \Rightarrow G_X^{k,(n)} \circ w_k^{(n)}
\] 
at an object $\big( (i_{01}, \dots , i_{0n}), \dots , (i_{k1}, \dots , i_{kn}) \big) \in \big( I^{(n)}\big)^{k+1} $ by replacing the permutations $\omega_{i_{j1} + \cdots + i_{jn}}$ by the permutations $\omega_{i_{j1}} \times  \cdots \times \omega_{i_{jn}}$ in the defining diagram for the natural transformation $w_k'$. The dihedral structure map is the composition
\[
w^{(n)}_k :\hocolim_{\big( I^{(n)}\big)^{k+1} }  \ G_X^{k,(n)} \xrightarrow{\big(w^{(n)}_k\big)'} \hocolim_{\big( I^{(n)}\big)^{k+1} } \ G_X^{k,(n)}  \circ w^{(n)}_k \xrightarrow{\ind_{w^{(n)}_k}}  \hocolim_{\big( I^{(n)}\big)^{k+1} } \ G_X^{k,(n)}. 
\]
Let $\THR^{(n)}(A;D;X)$ denote the geometric realization of $\THR(A,D;X)[-]$.

An order preserving inclusion $\iota: (m) \hookrightarrow (n)$ induces a functor 
\[
\iota: \big(I^{(m)}\big)^{k+1} \to \big(I^{(n)}\big)^{k+1}
\]
by inserting the initial element $0\in \Ob(I)$ into the added coordinates, which in turn induces a map 
\[
\iota: \THR^{(m)}(A;D;X) \to \THR^{(n)}(A;D;X).
\]
When $m\geq 1 $ it follows from the most general version of the Equivariant Approximation Lemma as stated in \cite[Prop. 2.7]{Hoe} that the map $\iota$ induces isomorphisms on $\pi_*^H(-)$ for all finite subgroups $H\leq O(2)$. 
      
We define the symmetric orthogonal spectrum $\THR(A,D)$ as follows. Let $n$ be a non-negative integer, and let $V$ be a finite $O(2)$-representation. The $(n,V)$th space is defined to be
\[
\THR(A,D)(n,V)=\triangle^*\THR^{(n)}(A,D; S^n\wedge S^V)
\]      
where $\triangle: \Sigma_n \times O(2) \to \Sigma_n \times O(2) \times \Sigma_n \times O(2) $
            is the diagonal map. The action by the first $O(2)$-factor arises from the dihedral structure and the action by the second $O(2)$-factor is induced from the $O(2)$-action on $V$. The action by the first $\Sigma_n$-factor is induced from permutation action on $I^{(n)}$ and the action by the second $\Sigma_n$-factor is induced from the $\Sigma_n$-action on $S^n$ given by permuting the sphere coordinates. The spectrum structure maps and the unit maps are described in \cite[Appendix]{GH} and one can verify that they are $G$-equivariant. 
            
            To define multiplicative structure maps we first recall that the canonical map 
           \[
\hocolim_{ (I^{(n)})^{k+1}} G^{k, (n)}_X \wedge \hocolim_{ (I^{(m)})^{k+1}} G^{k, (m)}_Y  \xrightarrow{\cong} \hocolim_{ (I^{(n+m)})^{k+1}} G^{k, (n)}_X \wedge G^{k, (m)}_Y
           \]
           is a homeomorphism, when the spaces are given the compactly generated weak Hausdorff topology. Next we note that there are natural transformations             
\[
\mu_{n,X,m,Y}': G_X^{k,(n)} \wedge G_Y^{k,(m)} \Rightarrow G_{X \wedge Y}^{k,(n+m)}
\]    
given by smashing together the maps $f\in G_X^{k,(n)}$ and $g \in G_Y^{k,(m)}$ and composing with the multiplication maps in $A$. The composition of the canonical map and the map induced by the natural transformation $\mu'$ commutes with the dihedral structure maps, the $\Sigma_n\times \Sigma_m$-action, and the $O(2)$-action from $X$ and $Y$. Given $n,m\geq 0$ and finite $O(2)$-representations $V$ and $W$ the multiplicative structure map is the geometric realization 
\[
\mu_{n,V,m,W}: \THR(A,D)(n,V) \wedge \THR(A,D)(m,W) \to \THR(A,D)(n+m,V\oplus W).
\]

   \section{The components of the $G$-geometric fixed points}
This section is devoted to the proof of Theorem \ref{A} in the introduction. We start by introducing some notation. We define the $G$-spheres $S^{1,0}=S^{\R}$ and $S^{1,1}=S^{i\R}$ to be the pointed $G$-spaces given by the one point compactifications of the $1$-dimensional trivial representation and sign representation, respectively. More generally, we set
\[
S^{p,q}=\big( S^{1,0}\big)^{\wedge (p-q)} \wedge \big( S^{1,1}\big)^{\wedge (q)} 
\]
for integers $p\geq q \geq 0$.

Let $EG$ be the free contractible $G$-CW-complex
\[
EG:=\bigcup_{n=0}^{\infty}S\big(\bigoplus_{j=1}^n i\mathbb{R}\big),
\] 
where $S(\oplus_{j=1}^n i\mathbb{R})$ denotes the unit sphere in $\oplus_{j=1}^n i\mathbb{R}$. We denote by $\widetilde{E}G$ the reduced mapping cone of the based $G$-map $EG_+ \to S^0$ which collapses $EG$ to the non-basepoint, hence 
\[
\widetilde{E}G=\colim_{k \to \infty} S^{k,k}.
\] 
If $X$ is an orthogonal $G$-spectrum, then the derived $G$-fixed points of $\widetilde{E}G \wedge X$ is a model for the derived $G$-geometric fixed point of $X$; see \cite[Prop. 4.17]{MM}. Consider the inclusion $S^{n,n} \to S^{n+1,n+1}$. There are canonical homeomorphisms of $G$-spaces
\begin{align*}
S^{n+1,n+1}/S^{n,n} \xrightarrow{\cong }\Sigma S^{n,n} \wedge G_+  \xrightarrow{\cong} S^{n+1} \wedge G_+,
\end{align*} 
where the first map is described in \cite[Lemma A.4]{Hoe} and the second map untwists the $G$-action, that is the map is given by $(x,g) \mapsto (g^{-1}x,g).$
Thus there are cofiber sequences of based $G$-CW-complexes for $n\geq 0$: 
\[
S^{n,n} \to S^{n+1,n+1} \to G_+ \wedge S^{n+1}.
\]
We smash the cofiber sequence with the orthogonal $G$-spectrum $X$ and obtain a long exact sequence of $G$-stable homotopy groups, which contains the segment
\begin{align*}
\cdots \to \pi_{-n}(X) \xrightarrow{} \pi_0^{G}(S^{n,n} \wedge X) \to \pi_0^{G}(S^{n+1,n+1} \wedge X)  
\to \pi_{-n-1}(X) \to \cdots
\end{align*}
If $X$ is connective, then the inclusion $S^{n,n} \to S^{n+1,n+1}$ induces an isomorphism 
\[
\pi_0^{G}(S^{n,n} \wedge X)
\xrightarrow{\cong} \pi_0^{G}(S^{n+1,n+1}\wedge X )
\]
for $n\geq 1$ and, in particular, the inclusion $S^{1,1} \to \widetilde{E}G$ induces an isomorphism 
\[
\pi_0^{G}(S^{1,1} \wedge X) \xrightarrow{\cong} \pi_0^G(\widetilde{E}G\wedge X).
\]
We are now ready to prove Theorem \ref{A} from the introduction. 

\begin{proof}[Proof of Theorem \ref{A}]
Let $(R, \alpha)$ denote a ring $R$ with an anti-involution $\alpha$. By the discussion above, the components of the derived $G$-geometric fixed points can be calculated as the homotopy group $\pi_0^{G}(S^{1,1} \wedge \THR(R, \alpha)) $. The actions of smashing an orthogonal $G$-spectrum with $S^{1,1}$ and shifting the spectrum by $i\mathbb{R}$ yield canonically $\pi_*$-isomorphic  $G$-spectra. It follows from Lemma \ref{fibrant} that we have an isomorphism of abelian groups:
\begin{align*}
\pi_0^{G}(S^{1,1} \wedge \THR(R, \alpha))  \cong \pi_0 \big( \THR(R, \alpha)(i\mathbb{R})^G \big).
\end{align*}
Segal's real subdivision described in Remark \ref{sub} provides a homeomorphism 
\[
D^e:  \vert \big( \sd^e \THR(R, \alpha; S^{1,1})[-] \big)^G \vert \xrightarrow{\cong} \THR(R, \alpha)(i\mathbb{R})^G.
\]
By \cite[Lemma 11.11]{May72}, we can calculate the group of components of the left hand side as the quotient of $\pi_0\big(\sd^e \THR(R, \alpha; S^{1,1})[0]^G\big)$ by the equivalence relation generated by $d^e_0(x) \sim d^e_1 (x)$ for all $x\in \pi_0\big( \sd^e \THR(A; S^{1,1})[1]^G\big)$. It follows from Lemma \ref{fixed points} that the diagram
 \[\begin{tikzpicture}
\setlength{\TMP}{3pt}
\node(a){$\sd^e \THR(R, \alpha; S^{1,1})[1]^G$};
\node(b)[right of=a, node distance = 5.2cm]{$\sd^e \THR(R, \alpha; S^{1,1})[0]^G$};
\draw[->] ([yshift=\TMP]a.east) to node [above]{$d_0^e$} ([yshift=\TMP]b.west);
\draw[->] ([yshift=-\TMP]a.east) to node [below]{$d_1^e$} ([yshift=-\TMP]b.west);

\end{tikzpicture}
\]
is homeomorphic to the left hand part of the homotopy commutative diagram
 \[
 \begin{tikzpicture}
 \setlength{\TMP}{3pt}
\node(a){$\displaystyle \hocolim_{I^{G} \times I \times I^{G}}  (G^{3}_{X}\circ \triangle^e)^{G}$};
\node(b)[below of=a, node distance = 2cm]{$\displaystyle  \hocolim_{I^{G}\times I^{G}}  (G^{1}_{X}\circ \triangle^e)^{G}$};
\node(c)[right of=a, node distance = 7.3cm]{$\Map(S^i \wedge S^n \wedge S^j \wedge S^n, HR_i \wedge HR_n \wedge HR_j \wedge HR_n \wedge S^{1,1} )^G$};
\node(d)[below of=c, node distance = 2cm]{$\Map(S^{n+i+n} \wedge S^{n+j+n}, HR_{n+i+n} \wedge HR_{n+j+n} \wedge S^{1,1} )^G.$};

\draw[->] ([xshift=\TMP]a.south) to node [right]{$d_0 \circ d_3$} ([xshift=\TMP]b.north);
\draw[->] ([xshift=-\TMP]a.south) to node [left]{$d_1 \circ d_2$} ([xshift=-\TMP]b.north); 

\draw[->] ([xshift=\TMP]c.south) to node [right]{$ \incl_j\circ d'_0 \circ d'_3$} ([xshift=\TMP]d.north);
\draw[->] ([xshift=-\TMP]c.south) to node [left]{$\incl_i \circ  d'_1 \circ d'_2$} ([xshift=-\TMP]d.north); 

\draw[->] (c) to node [left]{} (a);
\draw[->] (d) to node [left]{} (b);

\end{tikzpicture}
\]
Here $\incl_j$ is the image of the functor $G^1_{S^{1,1}} \circ \triangle^e$ at the morphism $(\id_{n+i+n}, \incl_{\text{mid}})$ in $I^G \times I^G$ where $\incl_{\text{mid}}$ is middle inclusion $j \to n+j+n$. The map $\incl_j$ is defined analogously. The horizontal maps can be made as connected as desired by choosing $i$, $j$ and $n$ big enough by the Equivariant Approximation Lemma \ref{EAL}. We fix a choice of $i$, $j$ and $n$ such that the horizontal maps are $0$-connected and for simplicity we choose $i$ and $j$ to be even.

If $X$ and $Y$ are $G$-CW-complexes, then the inclusion of fixed points $g: X^G \hookrightarrow X$ induces a fibration
\[g_*: \Map_G(X,Y)  \to \Map_G(X^G,Y)=\Map(X^G, Y^G)\]
with fiber $\Map_G(X/X^G,Y)$. It follows from \cite[Prop. 2.7]{Ad} that the connectivity of the fiber can be estimated as follows:
\[
\conn(\Map_G(X/X^G,Y)) \geq \underset{K \in \{e, G\}}{\min} (\conn(Y^K)-\dim((X/X^G)^K)).
\]
In the case at hand, the inclusions of $G$-fixed points
\[
\big(S^i \wedge S^n \wedge S^j \wedge S^n \big)^G \hookrightarrow S^i \wedge S^n \wedge S^j \wedge S^n
\] 
and 
\[
\big(S^{n+i+n} \wedge S^{n+j+n} \big)^G \hookrightarrow S^{n+i+n} \wedge S^{n+j+n}
\] 
induce fibrations with $0$-connected fibers, so the right hand part of the diagram evaluated at $\pi_0$ is isomorphic to the diagram
 \[
 \begin{tikzpicture}
  \setlength{\TMP}{3pt}
\node(a){$\pi_0 \Big( \Map(S^{\frac{i}{2}} \wedge S^n \wedge S^{\frac{j}{2}}, HR_i^{H\alpha \circ  \omega} \wedge HR_n \wedge HR_j^{H\alpha \circ \omega} ) \Big)$};
\node(b)[below of=a, node distance = 1.8cm]{$\pi_0 \Big( \Map(S^{n+\frac{i}{2}+n} \wedge S^{n+\frac{j}{2}+n}, HR_{n+i+n}^{H\alpha \circ \omega} \wedge HR_{n+j+n}^{H\alpha \circ \omega}) \Big),$};

\draw[->] ([xshift=\TMP]a.south) to node [right]{$\tilde{d}_0$} ([xshift=\TMP]b.north);
\draw[->] ([xshift=-\TMP]a.south) to node [left]{$\tilde{d}_1$} ([xshift=-\TMP]b.north); 

\end{tikzpicture}
\]
where $\tilde{d}_0:= \pi_0^*(\incl_j\circ d'_0 \circ d'_3)$ and $\tilde{d}_1:= \pi_0^*(\incl_i \circ  d'_1 \circ d'_2)$. We have omitted the index on $H\alpha$ and $\omega$. The space $HR_n$ is $(n-1)$-connected and Dotto proves that the space $HR_{2n}^{\alpha \circ \omega}$ is $(n-1)$-connected; see \cite[Lemma 6.3.2]{Do}. It follows from the Hurewicz isomorphism and the Künneth formula that the diagram above is isomorphic to the diagram
 \[
 \begin{tikzpicture}\label{diagram}
  \setlength{\TMP}{3pt}
\node(c){$\pi_{\frac{i}{2}}(HR_i^{H\alpha \circ \omega}) \otimes \pi_{n}(HR_n) \otimes  \pi_{\frac{j}{2}}(HR_j^{H\alpha \circ \omega})$};
\node(d)[below of=c, node distance = 1.8cm]{$\pi_{n+\frac{i}{2}}(HR_{n+i+n}^{H\alpha \circ \omega}) \otimes \pi_{n+\frac{j}{2}}(HR_{n+j+n}^{H\alpha \circ \omega}).$};

\draw[->] ([xshift=\TMP]c.south) to node [right]{$\tilde{d}_0$} ([xshift=\TMP]d.north);
\draw[->] ([xshift=-\TMP]c.south) to node [left]{$\tilde{d}_1$} ([xshift=-\TMP]d.north); 

\end{tikzpicture}
\]
The homotopy groups $\pi_n(HR_{2n}^{\alpha \omega})$ are independent of $n$ when $n\geq 1$ and we therefore calculate $\pi_1(HR_{2}^{\alpha \omega})$. The space $HR_2$ is the geometric realization of the simplicial set
\[
R[S^1[-]\wedge S^1[-]]/ R[*].
\] 
The action by $H\alpha \circ \omega$ is induced by a simplicial action where $\alpha$ acts on the $R$-label and $\omega$ acts by twisting the smash factors $S^1[-]\wedge S^1[-]$. Since taking fixed points of a finite group commutes with geometric realization, $HR_2^{H\alpha \circ \omega}$ is the geometric realization of the simplicial set
\[
\big( R[S^1[-]\wedge S^1[-]] \big)^{\alpha \circ \omega}/ R[*].
\] 
This is a simplicial abelian group and we may therefore calculate $\pi_1(HR_2^{H\alpha \circ \omega})$ as the first homology group of the associated chain complex. Recall that 
\[\Delta^1[k]=\text{Hom}_{\Delta}([k],[1])=\{x_0, x_1, \dots, x_{k+1}\}\]
where $ \# x_i^{-1}(0)=i$ and with the face maps given by
\[d_s(x_i)= \left\{ \begin{array}{rl}
 x_i &\mbox{ if $i\leq s $} \\
x_{i-1} &\mbox{ if $i>s.$}
       \end{array} \right. \]
The sphere $S^1[-]$ is defined to be the quotient $\Delta^1[-]/\partial \Delta^1[-]$. We have representatives of the simplices in $S^2[-]=S^1[-]\wedge S^1[-]$ as follows:
\begin{align*}
S^2[0]=\{x_0 \wedge x_0\}, \quad S^2[1]=\{x_0 \wedge x_0, \ x_1 \wedge x_1\}, \\
S^2[2]=\{x_0 \wedge x_0, \  x_1 \wedge x_1, \ x_2 \wedge x_2, \  x_1 \wedge x_2, \  x_2 \wedge x_1 \}.
\end{align*}
The associated chain complex of $\big( R[S^1[-]\wedge S^1[-]] \big)^{\alpha \circ \omega}/ R[*]$ begins with the sequence
\[
\cdots \to (R\cdot (x_1 \wedge x_1) \oplus R \cdot (x_1 \wedge x_2) \oplus R \cdot (x_2 \wedge x_1))^{\alpha \circ \omega} \xrightarrow{d} R^{\alpha} \cdot (x_1 \wedge x_1) \xrightarrow{} 0.
\]
Since $d(x_1 \wedge x_1)=0$, the first homology group is the cokernel of the map
\[
(R\cdot (x_1 \wedge x_2) \oplus R\cdot (x_2 \wedge x_1))^{\alpha \circ \omega} \xrightarrow{d} R^{\alpha}\cdot (x_1 \wedge x_1),
\]
where $d(x_1 \wedge x_2)=d(x_2 \wedge x_1)=-(x_1 \wedge x_1)$. The map 
\[
R \xrightarrow{} (R\cdot (x_1 \wedge x_2) \oplus R\cdot (x_2 \wedge x_1))^{\alpha \circ \omega} 
\]
which sends $r$ to the tuple $(-r,-\alpha(r))$ is an isomorphism, and $d$ corresponds to the norm map $N: R \to R^{\alpha}$ given by 
\[
N(r)=r +\alpha(r)
\] 
under this isomorphism. We conclude that $\pi_1(HR_2^{H\alpha \circ \omega})\cong R^{\alpha}/N(R)$. The diagram from before can now be identified with
 \[
 \begin{tikzpicture}
   \setlength{\TMP}{3pt}
\node(c){$R^{\alpha}/N(R)\otimes R \otimes  R^{\alpha}/N(R)$};
\node(d)[below of=c, node distance = 1.5cm]{$R^{\alpha}/N(R) \otimes R^{\alpha}/N(R)$};

\draw[->] ([xshift=\TMP]c.south) to node [right]{$\tilde{d}_0$} ([xshift=\TMP]d.north);
\draw[->] ([xshift=-\TMP]c.south) to node [left]{$\tilde{d}_1$} ([xshift=-\TMP]d.north); 

\end{tikzpicture}
\]
where $ \tilde{d}_0(r \otimes s \otimes t )=\alpha(s)rs \otimes t$ and $ \tilde{d}_1(r \otimes s \otimes t )=r \otimes st\alpha(s)$. It follows that
\[
\pi_0 \big( (\THR(R, \alpha)^c)^{gG} \big) \cong (R^{\alpha}/N(R)\otimes_{\Z} R^{\alpha}/N(R))/I,
\]
where $I$ denotes the subgroup generated by the elements $\alpha(s)rs \otimes t-r \otimes st\alpha(s)$ for all $s \in R$ and $r,t \in R^{\alpha}$. This completes the proof of Theorem \ref{A}.
\end{proof}

\begin{remark}
When $R$ is a commutative ring, then $\THR(R, \alpha)$ has the homotopy type of an $O(2)$-ring spectrum, hence the components of the $G$-geometric fixed points have a ring structure. We note that in this case $R^\alpha$ is a subring of $R$ and $N(R)$ is an ideal in $R^{\alpha}$. Furthermore the subgroup $I$ generated by the elements $\alpha(s)rs \otimes t-r \otimes st\alpha(s)$ for all $s \in R$ and $r,t \in R^{\alpha}$ is an ideal. It is generated as an ideal by the elements $ \alpha(r) \cdot r \otimes 1 - 1\otimes r\cdot \alpha(r) $ for all $r \in R$. These observations give 
\[
(R^{\alpha}/N(R)\otimes_{\Z} R^{\alpha}/N(R))/I
\]
 a natural ring structure.
\end{remark}

\begin{remark}\label{limit}
Let us consider the case where $R$ is a commutative ring with the identity serving as anti-involution. The functor
\[
R \mapsto \pi_0 \big(( \THR(R, \id)^c)^{gG} \big),
\]
considered as a functor from the category of commutative rings to the category of sets, is not representable. This rules out the possibility that the ring $\pi_0 \big(( \THR(R, \id)^c)^{gG} \big)$ is a ring of Witt vectors as defined by Borger in \cite{Borger}. For example, the functor does not preserve finite products. Indeed, in this case we have an isomorphisms of abelian groups
\[
\pi_0 \big(( \THR(R, \id)^c)^{gG} \big)\cong (R/2R \otimes R/2R)/I
\]
where $I$ is the ideal in the ring $R/2R \otimes R/2R$ generated as an ideal as follows
\[
I=\left( x^2 \otimes 1 - 1 \otimes x^2 \mid x \in R/2R \right).
\]
We consider the product $\mathbb{F}_2[x] \times \mathbb{F}_2[y]$. We have  a commutative diagram where the top right corner is the functor applied to the product ring and the lower right corner is the product of the functor applied to each factor. The horizontal maps are surjective quotient maps and the vertical maps are induced by the projections.
 \[
 \begin{tikzpicture}
\node(a){$\big (\mathbb{F}_2[x] \times \mathbb{F}_2[y]\big) \otimes \big(\mathbb{F}_2[x] \times \mathbb{F}_2[y]\big)$}; 
\node(b)[right of = a,node distance = 7cm]{$\Big( \big(\mathbb{F}_2[x] \times \mathbb{F}_2[y]\big) \otimes \big(\mathbb{F}_2[x] \times \mathbb{F}_2[y]\big)\Big) /I$};
\node(e)[below of = a,node distance = 1.5cm]{$ \big( \mathbb{F}_2[x] \otimes \mathbb{F}_2[x] \big) \times  \big( \mathbb{F}_2[y] \otimes \mathbb{F}_2[y]\big)$};
\node(f)[right of = e,node distance = 7cm]{$ \big(\mathbb{F}_2[x] \otimes \mathbb{F}_2[x]\big)/I \times \big( \mathbb{F}_2[y] \otimes \mathbb{F}_2[y]\big)/I$};
\draw[->>](a) to node [above]{} (b);
\draw[->>](e) to node [left]{} (f);
\draw[->](a) to node [right]{} (e);
\draw[->](b) to node [above]{} (f);
\end{tikzpicture}
 \]
 The claim is that the right vertical map is not a bijection. The left vertical map takes the element $(x,0)\otimes (0,y)$ to zero. The ideal 
 \[
 I \subset \big(\mathbb{F}_2[x] \times \mathbb{F}_2[y]\big) \otimes \big(\mathbb{F}_2[x] \times \mathbb{F}_2[y]\big)
 \] 
 is generated by the elements $\big(p(x)^2,q(y)^2 \big)\otimes 1 - 1 \otimes  \big(p(x)^2,q(y)^2\big)$, where $p(x)$ and $q(y)$ are polynomials. Hence $(x,0)\otimes (0,y) \notin I$ and the right vertical map is not injective. 
\end{remark}

In some cases, the calculation of the components of the $G$-geometric fixed points immediately leads to a calculation of the components of the $G$-fixed points as a ring. We have a cofibration sequence of $G$-spaces
\[
G_+ \to S^0 \to S^{1,1}.
\]
We smash the cofibration sequence with the spectrum $\THR(R, \alpha)$ and obtain a long exact sequence of $G$-stable homotopy groups which begins with the sequence 
\begin{align}\label{exact sequence}
\cdots \to \pi_0(\THR(R, \alpha)) \xrightarrow{V_{e}^{G}} \pi_0^{G}(\THR(R, \alpha)) \to \pi_0^G \big(S^{1,1} \wedge \THR(R, \alpha) \big) \to 0,
\end{align}
where $V_e^G$ denotes the transfer map; see \cite[Lemma 2.2]{HM97} for an identification of the induced map in the long exact sequence and the transfer map. We let 
\[
F_e^G: \pi_0^{G}(\THR(R, \alpha))  \to \pi_0(\THR(R, \alpha)) 
\]
denote the restriction map, which is a ring map. By the double coset formula 
\[F_e^G \circ V_{e}^{G} =N.\]

\begin{example}
If $R$ is a commutative ring with $\frac{1}{2} \in R$ and the identity serving as anti-involution, then it follows from the formula $F_e^G \circ V_{e}^{G} =2 \cdot \id$, that $V_{e}^{G}$ is injective. Since the components of the $G$-geometric fixed points vanish, it follows from the exact sequence \eqref{exact sequence} that 
 \[
F_e^G : \pi^{G}_0(\THR(R, \id)) \to  \pi_0(\THR(R, \id))=R\cdot 1
\] 
is a ring isomorphism.
\end{example}

\begin{remark} \label{witt}
Hesselholt and Madsen prove in \cite{HM97} that there is a canonical ring isomorphism identifying $ \pi_0\THH(R)^{C_{p^n}}$ with the $p$-typical Witt vectors of length $n+1$, when $R$ is a commutative ring and $p$ is a prime. Dress and Siebeneicher introduced a Witt vector construction in \cite{DS}, which is carried out relative to any given profinite group, and the $p$-typical Witt vectors of length $n+1$ are the Witt vectors constructed relative to $C_{p^n}$, i.e. $ \pi_0\THH(R)^{C_{p^n}} \cong W_{C_{p^n}}(R)$. If we let $D_{p^n}$ denote the dihedral group of order $2p^n$, then it is tempting to expect that the ring $\pi_0\THR(R, \id)^{D_{p^n}} $ can be identified with the Witt vectors $W_{D_{p^n}}(R)$ when $R$ is commutative but the example above shows that this is not the case. If $p>2$, then $W_G(\mathbb{F}_p)$ is isomorphic to $\mathbb{F}_p \times \mathbb{F}_p$, but $ \pi_0(\THR(\mathbb{F}_p, \id))^G$ is isomorphic to $\mathbb{F}_p$ by the example above.
\end{remark}

\begin{example}
We consider the example $(\Z,\id)$. Since $F_e^G \circ V_{e}^{G} =2\cdot \id$ and there is no $2$-torsion in $\Z$, the transfer map is injective and the long exact sequence \eqref{exact sequence} gives rise to a short exact sequence 
\[
0 \to \pi_0(\THR(\Z, \id)) \xrightarrow{V_e^G} \pi_0^G(\THR(\Z, \id)) \to \mathbb{Z}/2 \to 0,
\]
where $\pi_0(\THR(\Z))= \Z \cdot 1$. Since $\operatorname{Ext}_{\Z}^1(\Z/2\Z,\Z)=\Z/2\Z$, there are two possibilities for what this short exact sequence can look like, when considered as short exact sequence of abelian groups. The first possibility is 
\[
0 \to \Z \xrightarrow{V} \Z \times \Z/2\Z \to \mathbb{Z}/2\Z \to 0,
\]
with $V(x)=(x,0)$, hence $F(x,0)=2x$. Since $\mathbb{Z}$ is torsion free, $F(0,y)=0$, hence $F(x,y)=2x$. But then the pre-image of the unit is empty, which is a contradiction, since $F$ is a ring map. The short exact sequence must therefore be of the form
\[
0 \to \Z \xrightarrow{V} \Z \to \mathbb{Z}/2\Z \to 0,
\]
with $V=2 \cdot \id $ and $F=\id$. It follows that
\[F_e^G: \pi_0^G(\THR(\Z, \id)) \xrightarrow{} \pi_0(\THR(\Z, \id)) = \mathbb{Z} \cdot 1
\]
is a ring isomorphism.
\end{example}

\end{document}